\documentclass[11pt,a4paper]{article}
\usepackage{times}
\usepackage{tgtermes}
\usepackage[T1]{fontenc}
\usepackage[utf8]{inputenc}
\usepackage[english]{babel}
\usepackage{fancyhdr}
\usepackage{setspace}
\usepackage[backend=bibtex,sorting=nyt,maxnames=10]{biblatex}
\bibliography{cfm_2025_crespo_et_al.bib}
\usepackage{tikz}
\usetikzlibrary{arrows, positioning}
\usepackage{amsmath}
\usepackage{amsthm}
\usepackage{amsfonts}
\usepackage{hyperref}
\usepackage{cleveref}
\usepackage{amssymb}
\usepackage{mathtools}
\usepackage{multicol}
\usepackage{fancyhdr}

\newcommand{\op}[1]{\ensuremath{\operatorname{#1}}}
\newcommand{\ie}{\textit{i.e.}}
\newcommand\at[1]{
{#1}{\raisebox{-0.5\height}{\(\left. \vphantom{#1} \right |\)}\vphantom{#1}}
}

\newcommand{\intcc}[2]{\ensuremath{\left[#1, #2\right]}}
\newcommand{\intoo}[2]{\ensuremath{\left]#1, #2\right[}}
\newcommand{\tridots}{\mathop{:\!\cdot}}

\newcommand{\eg}{\textit{e.g.}}

\newcommand{\Af}{\mathbb{C}}              
\newcommand{\Bf}{{\mathbb{D}}}   
\newcommand{\Cf}{{\mathbb{E}}}    
\newcommand{\Df}{\mathbb{F}}              
\newcommand{\Ef}{{\mathbb{G}}}   

\newcommand{\transverse}{\mathcal{F}^\mathbb{B}}

\newcommand{\ov}[1]{\overline #1}
\newcommand{\Id}{\mathrm{Id}}

\newtheorem{theo}[equation]{Theorem}

\usepackage{titlesec}
\usepackage[most]{tcolorbox}
\titleformat{\section}
  {\normalfont\fontsize{16}{20}\bfseries}{\thesection}{1em}{}
\titleformat{\subsection}
  {\normalfont\fontsize{16}{20}\bfseries}{\thesubsection}{1em}{}
\titlespacing*{\section}
{0pt}{2.ex plus 1ex minus .2ex}{.0ex}
\titlespacing*{\subsection}
{0pt}{0.ex}{.0ex}

\begin{document}
\pagestyle{empty}

\newcommand{\preprintfooter}{\vspace{-2.5em}\ \\
  \begin{tcolorbox}[colback=red!5!white, colframe=red!50!black, boxrule=0.5pt, arc=2mm, left=2mm, right=2mm, width=.9\linewidth]
    \centering\small\itshape
    \textcolor{red!70!black}{The present manuscript is a preprint pending peer review; please read with due scientific caution.}
  \end{tcolorbox}
}

\pagestyle{fancy}
\pagenumbering{arabic}
\lfoot{\preprintfooter}
\rfoot{\thepage}

\begin{center}
\begin{spacing}{2.05}
{\fontsize{20}{20}
\bf
A New Framework for Unidimensional Structures Based on Generalised Continua
}
\end{spacing}
\end{center}
\vspace{-1.25cm}
\begin{center}
{\fontsize{14}{20}
\bf
M. Crespo\textsuperscript{a,c}, G. Casale\textsuperscript{a,d}, L. Le Marrec\textsuperscript{a,e} and P. Neff\textsuperscript{b}\\
\bigskip
}
{\fontsize{12}{20}
a. Institut de Math\'ematiques de Rennes, University of Rennes, Rennes 263 avenue du G\'en\'eral Leclerc\\
b. Head of Chair for Nonlinear Analysis and Modelling, Faculty of Mathematics,\\
University of Duisburg-Essen, Thea-Leymann-Straße 9, Essen, Germany, patrizo.neff@uni-due.de\\
c. mewen.crespo@univ-rennes.fr\quad
d. guy.casale@univ-rennes.fr\quad
e. loic.lemarrec@univ-rennes.fr
}
\end{center}

\vspace{10pt}

{\fontsize{16}{20}
\bf
Abstract
}
\bigskip

\textit{
The present work introduces a family of beam models derived from a three-dimensional higher-order elasticity framework. By incorporating three kinematic fields -- the macroscopic displacement \(u\), the micro-distortion tensor \(P\), and the third-order tensor \(N\) -- the study systematically explores three regimes: holonomic, semi-holonomic, and non-holonomic. These regimes correspond to varying levels of kinematic constraints, ranging from classical elasticity to a fully relaxed model. The holonomic case reduces to a higher-order Euler--Bernoulli beam model, while the semi-holonomic case generalises the Timoshenko beam model. The non-holonomic case provides a unified framework that naturally incorporates both dislocations and disclinations. Furthermore, the holonomic and semi-holonomic models are shown to emerge as singular limits of the non-holonomic model by increasing specific penalty coefficients. Simplified ordinary differential equation systems are derived for specific cases, such as pure traction and bending, illustrating the practical applicability of the models. The results highlight the hierarchical structure of the proposed framework and its ability to capture material defects in beam-like structures.
}

\vspace{28pt}

{\fontsize{14}{20}
}
\bigskip

\section{Introduction}

\subsection{Historical perspective}

\renewcommand{\th}{\textsuperscript{th}}

The scientific interest for beams is centuries old, vastly predating modern continuum mechanics theories. Some place its start as early as the 15\th century\,\autocite{BallariniVincieulerbernoulliBeamTheory2003} with the work of Leonardo da Vinci (1452--1519). Although lacking the mathematical tools to formalise his insights, his sketches and notes in his \textit{Codex Madrid I \& II}\,\autocite{daVinciCodexMadridI} indicate a qualitative understanding of structural behaviours. For others, it starts in the 17\th century, with the work of Galileo Galilei (1564--1642) who made one of the first attempt to quantify the strength of beams, and in particular their failure under their own weight~\autocite{galilei1638discorsi}. However, his assumptions were flawed, particularly regarding the distribution of stress across a beam's cross-section. It is another century later, that Jacob Bernoulli (1655--1705) corrected Galileo's work~\autocite{BernoulliVeritableHypotheseResistance1705}, introducing the concept of curvature in the study of elastic bodies.\\

If Jacob Bernoulli laid the groundwork for understanding how beams bend, a comprehensive theory was not yet fully developed. Such a theory was only developed in the middle of the 18\th century by Daniel Bernoulli\footnote{Daniel Bernoulli was the nephew of Jacob Bernoulli.} (1700--1782) and Leonhard Euler (1707--1783)~\autocite{euler1744methodus,bernoulli1738hydrodynamica}. They introduced the principle that internal stresses in a beam are related to its curvature, paving the way for what is now known as the Euler--Bernoulli beam model. Importantly, the model relies on the following assumptions:
\begin{enumerate}
    \item the beam is slender, with its longitudinal length \(L\) significantly greater than its transversal dimensions;
    \item transversal sections remain plane and perpendicular to the neutral axis after deformation and
    \item only transversal loads are considered, meaning that the effects of shear deformation and rotatory inertia are ignored.
\end{enumerate}
The first hypothesis is common across all beam models and is sometimes used as a definition of a beam. Under the Euler--Bernoulli framework, the relationship between the bending moment \(M(X)\) and the beam's Euclidean curvature \(\frac{\mathrm{d}^2\, w}{\mathrm{d}X^2}\) is given by:
\begin{align}
    M\left( X \right) &= -E\,I\,\frac{\mathrm{d}^2\, w}{\mathrm{d}X^2},
\end{align}
where \(E\) is the Young's modulus -- Thomas Young (1773--1829) --, \(I\) is the second moment of area of the beam's cross-section and \(w(X)\) is the transverse displacement of the beam at position \(X \in \intcc 0L\). Combining this result with (the angular momentum version of) the second law of Isaac Newton (1642 -- 1727), one gets:
\begin{align}
    f(X)
        &= \frac{\mathrm{d}^2}{\mathrm{d} X^2}\left( E\,I\,\frac{\mathrm{d}^2\, w}{\mathrm{d}X^2} \right),
\end{align}
where \(f(X)\) is the distributed load per unit length acting on the beam. Although this model was first enunciated circa 1750, it was not applied on a large scale until the late 19\th century, notably with the development of Gustave Eiffel's (1832--1923) Tower in 1887. It then became a cornerstone of the second industrial revolution.\\

In the meantime, Joseph-Louis de Lagrange\footnote{Born under the italian name Giuseppe Luigi Lagrangia from French parents, he later took the French nationality in 1802} (1736--1813) published his pioneer work \textit{Mécanique Analytique}~\autocite{LagrangeMechaniqueAnalytiqueParis1787}. More than thirty years later, building upon Lagrange's analytical mechanics, Augustin-Louis Cauchy (1789--1857) published his work \textit{Sur l'équilibre et le mouvement intérieur des corps considérés comme des masses continues}, considered by many as the starting point of modern continuum mechanics. The latter is fundamentally tridimensional, while Euler--Bernoulli beam model is a unidimensional model.\\

A natural wish may therefore be to obtain the Euler--Bernouilli beam model as a special limit case of continuum mechanics and, based on this understanding, try to extend it. This quest proved to be rather tedious, with progress still being made nowadays. This is the focus of this paper.\\

In 1856, Adhémar Barré de Saint-Venant (1797--1886) published his "Mémoire sur la torsion des prismes"~\autocite{deSaintMemoireTorsionPrismes1856}. In this work, he laid one of the first stones by addressing the pure torsion (no bending nor compression) of homogeneous bars and derived a governing equation based on linear elasticity and continuity of stress. His formulation allowed for non-uniform distributions of shear stress in the cross-section of prismatic bars. More importantly, Saint-Venant introduced the principle -- now bearing his name -- stating that localised effects due to applied loads vanish at distances far from the point of application~\autocite{deSaint-VenantMemoireDiversGenres1865}. This principle provided a rigorous justification for employing simplified models in structural analysis and is used in most modern models.\\

In 1876, Leo Pochhammer (1841--1920) published a seminal paper where he derived the equations governing the propagation of longitudinal waves in an infinite, isotropic, elastic cylinder~\autocite{UeberFortpflanzungsgeschwindigkeitenKleiner1876}. This work was then enriched in 1889 by Charles Chree (1860--1928)~\autocite{ChreeEquationsIsotropicElastic1889}. The Pochhammer--Chree equations describe how stress waves propagate in cylindrical rods and were one of the first beam models using tridimensional continuum mechanics. However, this model only considers tension/compression and does not contain Euler--Bernoulli beam models, nor does it contain a practical generalisation.\\

In the early 20\th century, Stephen Prokofievich Timoshenko (1878--1972) successfully used tridimensional continuum mechanics to obtain a unidimensional beam model~\autocite{TimoshenkoCorrectionShearDifferential1921}. His model is a direct generalisation of Euler--Bernoulli beam model, where a new variable \(\theta\) is introduced. In static equilibrium, the governing equations are:
\begin{align}
    \label{eq_timoshenko_intro}
    f(X)
        &= \frac{\mathrm{d}^2}{\mathrm{d} X^2}\left( \kappa G A \left ( \frac{\mathrm{d} w}{\mathrm{d}X} - \theta \right ) \right),\\\nonumber
    0
        &= \frac{\mathrm{d} w}{\mathrm{d}X} - \theta + \frac{1}{\kappa G A} \frac{\mathrm{d}}{\mathrm{d}X}\left( EI\frac{\mathrm{d}\theta}{\mathrm{d} X} \right),
\end{align}
where \(A\) is the transversal cross-section area, \(G\) the shear modulus and \(\kappa\) is a shear coefficient obtained from the shape of the cross-section (\(\kappa=\frac 56\) for a rectangular section). One obtains Euler-Bernoulli beam theory by substituting \cref{eq_timoshenko_intro}.2 in to \cref{eq_timoshenko_intro}.1 and taking the limit \(G \to \infty\). This corresponds to not allowing any shearing of the transversal planes. That is, imposing \(\theta = \frac{\mathrm{d} w}{\mathrm{d}X}\). The model of \cref{eq_timoshenko_intro} is known as the Timoshenko or Timoshenko--Ehrenfest beam model -- from the name of Paul Ehrenfest (1880--1933), latter shown to have been an undeclared co-author~\autocite{ElishakoffWhoDevelopedSocalled2020}.\\

Parallelly, in his pioneer work of 1907~\autocite{VolterraLequilibreCorpsElastiques1907}, Vito Volterra (1860--1940) studied the limitations of the theory of continuum mechanics. By carefully cutting and gluing back rubber cylinders, he showed the existence of translational and rotational discontinuity. Such defects, called dislocations and disclinations\footnote{This is the modern nomenclature. In \autocite{VolterraLequilibreCorpsElastiques1907} Volterra originally uses the term "\textit{distorsion}", which was a tensor with six components, encapsulating both defects. In some text, disclinations are also called \text{disinclinations}~\autocite{FrankLiquidCrystalsTheory1958}.}, arise from the lack of regularity of the actual atomistic placement, which violates Schwarz equality. Importantly, he postulated that all defects of a material -- defined as a discrepancy between the material's true placement and the smooth elastic approximation -- can be obtained from combinations of dislocations and disclinations.\\

This early work proved the need for new model, incorporating the microscopic structure in their kinematics. Such a model was developed by the Cosserat brothers -- François Nicolas Cosserat (1852--1914) and Eugène-Maurice-Pierre Cosserat (1866--1931) -- in 1909, in their book \textit{Théorie des Corps Déformables}~\autocite{CosseratCosseratTheorieCorpsDeformables1909}. In particular, they propose a model in which a smooth macroscopic n-dimensional field from Cauchy's continuum mechanics is enriched with a field of orthonormal frames, prescribing the orientation of the microstructure at each macroscopic point.\\

In particular, the \(n=1\) case, coined \textit{ligne déformable} in \autocite{CosseratCosseratTheorieCorpsDeformables1909}, is a beam model, often referred to as the Cosserat beam. Crucially, by restricting these orthonormal frames to be planar rotations along the osculating plane -- \ie{} no screw twisting --, the kinematic of the Cosserat beam becomes equivalent to Timoshenko beam's kinematic. Furtheremore, similarly to the Timoshenko case, when the orthonormal frames correspond to the gradient of the macroscopic film, one retrieves the kinematic of Euler-Bernoulli beam theory.\\

The Cosserat brothers' model was latter geometrically formalised by \'Elie Cartan (1869--1951)~\autocite{CartanGeneralisationNotionCourbure1922} and further developed, in the remaining of the 20\th century, by Raymond Mindlin (1906--1987)~\autocite{MindlinMicrostructureLinearElasticity1964}, Ahmed Cemal Eringen (1921--2009)~\autocite{EringenMechanicsMicromorphicContinua1969,EringenMicrocontinuumFieldTheories1999,EringenTheoryMicropolarElasticity1968}, Richard Toupin (1926--2017)~\autocite{ToupinTheoriesElasticityCouplestress1964,ToupinElasticMaterialsCouplestresses1962} and many others, with the first existence theorem given by Patrizio Neff (1969--today)~\autocite{NeffExistenceMinimizersFinitestrain2006}. This lead to a theory commonly referred to as the theory of micromorphic media, where the microscopic placement does not have to be orthogonal anymore, but is rather given by an unconstrained matrix field~\(\Xi\).\\

\subsection{Approach in this work}

The theory of micromorphic media allows for dislocations to naturally arise, encapsulated in the non-triviality of \(\op{curl}\Xi\). However, its kinematic prohibits the apparition of disclinations. Indeed, the latter corresponds to a violation of Schwarz equality for the second order, however one has: \(\op{Curl}\nabla \Xi = 0\). One therefore needs to relax the second order, similarly to how the first order was relaxed when the microstructure was introduced. This work has been done in a previous article by the first authors of the present paper~\autocite{CrespoEtAlTwoScaleGeometricModelling2024}. Importantly, this model allows for disclinations to arise from the kinematics alone.\\

The model is built upon the following kinematic variables:
\begin{enumerate}
    \item \(u : \mathbb{B} \longrightarrow \mathbb{R}^3\) --- a \(3\)-dimensional displacement field defined on \(\mathbb{B}\),
    \item \(P : \mathbb{B} \longrightarrow \mathbb{R}^{3 \times 3}\) --- a matrix field representing the microstructural displacement,
    \item \(N : \mathbb{B} \longrightarrow \mathbb{R}^{3 \times 3 \times 3}\) --- a third-order tensor field capturing first-order microstructural changes,
\end{enumerate}

where \(\mathbb{B}\) is the material manifold (or body), \(\left\| u \right\| \ll 1\), \(\left\| P \right\| \ll 1\). The field \(\varphi(X) = X + u(X)\) corresponds to the smooth macroscopic deformation of classical continuum mechanics and the field \(\Id + P\) corresponds to the microscopic field \(\Xi\) introduced by the micromorphic theory. \(N\) is the new field added in \autocite{CrespoEtAlTwoScaleGeometricModelling2024} \(\left (\text{named \(\mathbf{F}_{\mathrm{h}}^{\mathrm{v}}\) there}\right )\). The associated constitutive coefficients are:
\begin{multicols}{2}
\begin{enumerate}
    \item \(\Af : \mathbb{B} \longrightarrow \mathbb{R}^{3\times 3 \times 3 \times 3}\),
    \item \(\Bf : \mathbb{B} \longrightarrow \mathbb{R}^{3\times 3 \times 3 \times 3\times 3 \times 3}\),
    \item \(\Cf : \mathbb{B} \longrightarrow \mathbb{R}^{3\times 3 \times 3 \times 3 \times 3\times 3 \times 3 \times 3}\),
    \item \(\Df : \mathbb{B} \longrightarrow \mathbb{R}^{3\times 3 \times 3 \times 3}\),
    \item \(\Ef : \mathbb{B} \longrightarrow \mathbb{R}^{3\times 3 \times 3 \times 3 \times 3 \times 3}\),
\end{enumerate}
\end{multicols}

and the associated force fields are:
\begin{multicols}{2}
\begin{enumerate}
    \item \(f_0 : \mathbb{B} \longrightarrow \mathbb{R}^3\),
    \item \(f_1 : \mathbb{B} \longrightarrow \mathbb{R}^{3 \times 3}\),
    \item \(f_2 : \mathbb{B} \longrightarrow \mathbb{R}^{3 \times 3 \times 3}\),
    \item \(T_0 : \partial\mathbb{B} \longrightarrow \mathbb{R}^3\),
    \item \(T_1 : \partial\mathbb{B} \longrightarrow \mathbb{R}^{3 \times 3}\),
    \item \(T_2 : \partial\mathbb{B} \longrightarrow \mathbb{R}^{3 \times 3 \times 3}\).
\end{enumerate}
\end{multicols}

The total energy of the system is then given by:
\begin{align}
    &\int_\mathbb{B} \op{sym}{P} : \Af : \op{sym}{P} + N \tridots \Bf \tridots N
    + \nabla N :: \Cf :: \nabla N\\\nonumber
    &\quad\quad + \left( \nabla u - \at{P}_{1\ldots n} \right) : \Df : \left( \nabla u - \at{P}_{1\ldots n} \right)
    + \left( \nabla P - \at{N}_{1 \ldots n} \right) \tridots \Ef \tridots \left( \nabla P - \at{N}_{1 \ldots n} \right)\\\nonumber
    &\quad\quad - f_0 \cdot u - f_1 : P - f_2 \tridots N\, \mathrm{d}X
    - \int_{\partial \mathbb{B}} T_0 \cdot u + T_1 : P + T_2 \tridots N\, \mathrm{d} X.
\end{align}

where \(2\left[ \op{sym}P \right]^i_j := P^i_j - P^j_i\) is the linearisation of \(\left( \Id + P \right)^\mathrm{T}\cdot\left( \Id + P \right)\) and \(2\left[ \op{Skew}N \right]^i_{jk} = N^i_{jk} - N^i_{kj}\). The dependency of the energy on \(\op{sym}P\), \(\nabla P\) \(N\), \(P - \nabla u\) and their gradients is required by objectivity, as derived in \autocite{CrespoEtAlTwoScaleGeometricModelling2024}. We make the choice to neglect most mixed terms as well as higher order terms in \(u\) and \(P\). Furthermore, in order that explicit analytical solutions may be found, we shall also simplify the constitutive tensors as follows:
\begin{align}
    \Af^{il}_{jm} &= a\,\delta^{il}\delta_{jm}, &
    \Bf^{ijlm}_{kn} &= b\,\delta^{il}\delta^{jm}\delta_{kn},&
    \Cf^{ij1lm1}_{kn} &= c\,\delta^{il}\delta^{jm}\delta_{kn},&\\\nonumber
    \Df^{11}_{ij} &= d\,\delta_{ij},&
    \Ef^{i1l1}_{jm} &= e\,\delta^{il}\delta_{jm}.
\end{align}
so that every contraction of a constitutive tensor becomes a norm (\eg{} \(\op{sym}P : \Af : \op{sym}P = a \left\| \op{sym}P \right\|^2\)). Technically, the Kronecker symbols should be replaced with the macroscopic, microscopic and mixed metrics of \cite{CrespoEtAlTwoScaleGeometricModelling2024}. However, in the coordinate system we use, those all reduce to identity matrices. For readability, Kroneckers will be omitted when their only effect is to change the variance of an index. Under those constitutive assumptions, the internal energy simplifies to:
\begin{align}
	&\int_\mathbb{B}
			\frac a4 \left( P^i_j + P^j_i \right)\left( P^i_j + P^j_i \right)
			+ b\, N^i_{jk}\,N^i_{jk}
    		+ c\,N^i_{jk,l}\,N^i_{jk,l}\\\nonumber
    &\quad\quad
			+ d \left( u^i_{,j} - P^i_j \right) \left( u^i_{,j} - P^i_j \right)
    		+ e \left( P^i_{j,k} - N^i_{jk} \right) \left( P^i_{j,k} - N^i_{jk} \right)\, \mathrm{d} X,
\end{align}
where Einstein summation convention is used (repeated indices are summed) and will be for the remainder of this article.\\

In the present paper, we present a family of beam models, obtained from the tridimensional model detailed above under some kinematic constraints. Crucially, those models are valid for bending, tension/compression, twisting or any other kind of deformation, as long as it remains in the realm of small deformations. This family is composed of three models, distinguished by their holonomy constraints. Namely, we shall look at the following set of constraints:

\begin{enumerate}
    \item \textbf{holonomic case --} the constraints \(\nabla u = \at{P}_{1\ldots n}\) and \(\nabla P = N\) are imposed. These imply, by the divergence theorem, that the auxiliary fields \(\Df = 0\), \(\Ef = 0\), \(f_1 = 0\) and \(f_2 = 0\). Under these assumptions, the model depends solely on the displacement field \(u\), and thus reduces to a third gradient theory of elasticity. When \(\Cf = 0\), this is a second gradient theory, as developed by Mindlin, Germain and others~\autocite{MindlinSecondGradientStrain1965,MindlinEshelFirstStraingradientTheories1968,GermainMethodePuissancesVirtuelles1973}.

    \item \textbf{semi-holonomic case --} in this intermediate case, only the constraint \(\nabla P = N\) is enforced, leading to \(\Ef = 0\) and \(f_2 = 0\). This formulation involves both \(u\) and \(P\) as degrees of freODEm, both independent from each other in the bulk. When in addition \(\Cf = 0\), the model corresponds to the classical micromorphic theory developed by Mindlin, Eringen, Toupin, and others\,\cite{MindlinMicrostructureLinearElasticity1964,EringenFoundationsMicropolarThermoelasticity1970,EringenMechanicsMicromorphicContinua1969,EringenTheoryMicropolarElasticity1968,EringenMicrocontinuumFieldTheories1999,EringenClausMicromorphicApproachDislocation1970,ToupinTheoriesElasticityCouplestress1964,ToupinElasticMaterialsCouplestresses1962}.

    \item \textbf{non-holonomic case --} no constraint is imposed. The field \(u\), \(P\) and \(N\) are supposed independent in the bulk.
\end{enumerate}

\Cref{sec_1d_in_3d} is dedicated to the construction of the beam models from the tridimensional model presented above. Using the transversal smallness of the beam, a kinematic hypothesis is introduced, allowing to rephrase the tridimensional model in terms of kinematically independent unidimensional fields. In \cref{sec_variational_principle}, variationnal calculus is used to obtain the Euler-Lagrange equations -- a complete system of ordinary second-order linear differential equation -- associated to static equilibrium in the non-holonomic ,semi-holonomic and holonomic case. Lastly, \cref{sec_study_system} is focussed on a qualitative and quantitative analysis of the aforementioned systems. In particular, simple excitation schemes, allowing for significant simplifications of the system, are highlighted and the semi-holonomic (resp. holonomic) system is proved to be the asymptotic limits of the non-holonomic (resp. semi-holonomic) system when \(e \to \infty\) (resp. \(d \to \infty\)).

\section{The 1D in 3D model}

\label{sec_1d_in_3d}

A beam is a medium in which one dimension is significantly larger than the others, negligible in comparison. The large dimension is referred to as the longitudinal direction, while the other are referred to as transversal. By convention, we fix the coordinate system so that the longitudinal coordinate corresponds to the first coordinate. Furthermore, \emph{while latin indices and exponents will be valued in \(\{1,2,3\}\), greek ones will be valued in \(\{2,3\}\).} We introduce the following notation:
\begin{align}
    \overline{A}\left( X^1 \right) &:= A\left( X^1, 0, 0 \right)
    &\text{for any tensor \(A\)}
\end{align}
Due to their small magnitude, the transversal coordinates will have a small impact on the displacement. One can therefore safely neglect higher order terms and approximate a tensor \(A\) by its truncated Taylor expansion at rank \(r \in \mathbb{N}\):
\begin{align}
    A\left( X^1, X^2, X^3 \right) &\simeq \overline{A} + \overline{A_{,\alpha}} \, X^\alpha + \frac 1{2}\, \overline{A_{,\alpha_1\alpha_2}} \, X^{\alpha_1}\, X^{\alpha_2} + \cdots + \frac 1{r!}\, \overline{A_{,\alpha_1\cdots\alpha_r}} \, X^{\alpha_1} \cdots X^{\alpha_r}.
\end{align}
One can then replace the 3-dimensional tensor field \(A\) by \(r+1\) unidimensional tensor fields. Furthermore, if the transversal dimensions are extremely small, then one can effectively assimilate them with the microscopic dimensions. This interpretation motivates the following kinematic assumption:
\begin{align}
    \label{hypo_1d_3d}
    u^i_{,\alpha} &\simeq P^i_\alpha,&
    P^i_{j,\alpha} &\simeq N^i_{j\alpha}.
\end{align}

We then choose to take the smallest possible rank for the truncated Taylor expansion of \(N\). This corresponds to assuming
\begin{align}
    \label{hypo_1d_3d_non_holo}
    N^i_{jk,\alpha} &\simeq 0~.
\end{align}

Note that those approximations immediately imply:
\begin{align}
    N^i_{jl}\left( X^1, X^2, X^3 \right) &\simeq \overline{N}^i_{jl}\left( X^1\right),\\\nonumber
    P^i_j\left( X^1, X^2, X^3 \right) &\simeq \overline{P}^i_j\left( X^1\right) + \overline{N}^i_{j\alpha}\left( X^1 \right) X^\alpha,\\\nonumber
    u^i\left( X^1, X^2, X^3 \right) &\simeq \overline{u}^i\left( X^1 \right) + \overline{P}^i_\alpha\left( X^1 \right) X^\alpha.
\end{align}

The assumptions of \cref{hypo_1d_3d,hypo_1d_3d_non_holo} effectively kill all contributions of the transversal derivatives in the internal energy, which becomes:
\begin{align}
    \label{eq_energy_1d_raw}
    &\int_\mathbb{B}
        \frac a4 \left ( \ov{P^i_j} + \ov{N^i_{j\alpha}}\, X^\alpha + \ov{P^j_i} + \ov{N^j_{i\beta}}\, X^\beta \right )\left ( \ov{P^i_j} + \cdots + \ov{N^j_{i\gamma}}\, X^\gamma \right )
        + b\,\ov{N^i_{jk}}\,\ov{N^i_{jk}}\\\nonumber
    &\qquad
        + c\, \ov{N^i_{jk,1}}\, \ov{N^i_{jk,1}}
        + d \left ( \ov{u^i_{,1}} + \ov{P^i_{\alpha,1}}\, X^\alpha
            - \ov{P^i_1} - \ov{N^i_{1\beta}}\, X^\beta \right ) \left(\ov{u^i_{,1}} + \cdots - \ov{N^i_{1\mu}}\, X^\mu  \right)\\\nonumber
    &\qquad
        + e \left ( \ov{P^i_{j,1}} + \ov{N^i_{j\alpha,1}}\, X^\alpha - \ov{N^i_{j1}} \right )\left ( \ov{P^i_{j,1}} + \ov{N^i_{j\beta,1}}\, X^\beta - \ov{N^i_{j1}} \right )\\\nonumber
    &\quad\quad
    - f_0 \cdot u
    - f_1 : P
    - f_2 \tridots N\, \mathrm{d}X
    - \int_{\partial \mathbb{B}} T_0 \cdot u + T_1 : P + T_2 \tridots N\, \mathrm{d} X\,.
\end{align}

In order to simplify the above integrals, we choose to model the material as \(\mathbb{B} := \overline{\mathbb{B}} \times \transverse\) where \(\overline{\mathbb{B}} := \intcc 0L\), \(L \in \mathbb{R}^*_+\) is the longitudinal characteristic length and \(\transverse\) is the bidimensional shape of the transversal slices\footnote{The \(\mathcal{F}\) stands for \textit{fibre}, which is the mathematical name of what beam theorists usually call \textit{transversal section}. Similarly, the engineers' \textit{central fibre} is mathematically coined a \textit{section}. We shall use the terms \textit{transversal slice} (or plane) and \textit{central curve} respectively in order to avoid such confusion in the nomenclature.} (any connected bounded locally \(\mathcal{C}^2\) region of \(\mathbb{R}^2\)). We make the hypothesis that the coordinate system is chosen such that \(\left( X^1,  0, 0 \right)\) is the geometric center of \(\left\{ X^1 \right\} \times \transverse\). This means that most cross-terms vanish and the internal energy becomes:
\begin{align}
    &\int_\mathbb{B}
        \frac a4 \left ( \ov{P^i_j} + \ov{P^j_i} \right )\left ( \ov{P^i_j} + \ov{P^j_i} \right )
         + \frac a4 \left (\ov{N^i_{j\alpha}} + \ov{N^j_{i\alpha}} \right )\left (\ov{N^i_{j\beta}} + \ov{N^j_{i\beta}} \right ) X^\alpha\, X^\beta
         + b\, \ov{N^i_{jk}}\,\ov{N^i_{jk}}
    \\\nonumber
    &\qquad
        + c\, \ov{N^i_{jk,1}}\, \ov{N^i_{jk,1}}
        + d \left ( \ov{u^i_{,1}} - \ov{P^i_1} \right ) \left ( \ov{u^i_{,1}} - \ov{P^i_1} \right )
        + d \left ( \ov{P^i_{\alpha,1}} - \ov{N^i_{1\alpha}} \right )\left ( \ov{P^i_{\beta,1}} - \ov{N^i_{1\beta}} \right ) X^\alpha\, X^\beta
    \\\nonumber
    &\qquad
        + e \left ( \ov{P^i_{j,1}} - \ov{N^i_{j1}} \right )\left ( \ov{P^i_{j,1}} - \ov{N^i_{j1}} \right )
        + e \,\ov{N^i_{j\alpha,1}}\,\ov{N^i_{j\beta,1}}X^\alpha\, X^\beta
    \, \mathrm{d}X\,.
\end{align}

Then, by partially integrating the bulk integral of \cref{eq_energy_1d_raw} along the transversal coordinate, one  obtains an integral along the unidimensional body \(\overline{\mathbb{B}}\):
\begin{align}
    \label{eq_energy_non_holo_1d}
    &\int_{\overline{\mathbb{B}}}
        \frac a4 \left ( \ov{P^i_j} + \ov{P^j_i} \right )\left ( \ov{P^i_j} + \ov{P^j_i} \right )
         + \frac a4 \left (\ov{N^i_{j\alpha}} + \ov{N^j_{i\alpha}} \right )\left (\ov{N^i_{j\beta}} + \ov{N^j_{i\beta}} \right ) \mathrm{I}^{\alpha\beta}
         + b\, \ov{N^i_{jk}}\,\ov{N^i_{jk}}
    \\\nonumber
    &\qquad
        + c\, \ov{N^i_{jk,1}}\, \ov{N^i_{jk,1}}
        + d \left ( \ov{u^i_{,1}} - \ov{P^i_1} \right ) \left ( \ov{u^i_{,1}} - \ov{P^i_1} \right )
        + d \left ( \ov{P^i_{\alpha,1}} - \ov{N^i_{1\alpha}} \right )\left ( \ov{P^i_{\beta,1}} - \ov{N^i_{1\beta}} \right ) \mathrm{I}^{\alpha\beta}
    \\\nonumber
    &\qquad
        + e \left ( \ov{P^i_{j,1}} - \ov{N^i_{j1}} \right )\left ( \ov{P^i_{j,1}} - \ov{N^i_{j1}} \right )
        + e \,\ov{N^i_{j\alpha,1}}\,\ov{N^i_{j\beta,1}}\mathrm{I}^{\alpha\beta}
    \, \mathrm{d}X^1\,,
\end{align}
where \(\mathrm{I}^{\alpha\beta} = \int_\transverse X^\alpha\, X^\beta\, \mathrm{d} X\), which is diagonal and, for simplicity, assumed constant on \(\overline{\mathbb{B}}\). Interestingly, one obtains an antisymmetric coupling involving
\begin{align}
    d \left( \ov{P^i_{\alpha,1}} - \ov{N^i_{1\alpha}} \right) \left( \ov{P^i_{\beta,1}} - \ov{N^i_{1\beta}} \right) \mathbb{I}^{\alpha\beta}
        &\simeq \frac d2 \left( \ov{P^i_{\alpha,1}} - \overline{P^i_{1,\alpha}} \right) \left( \ov{P^i_{\beta,1}} - \overline{P^i_{1,\beta}} \right) \mathbb{I}^{\alpha\beta}\\\nonumber
        &\simeq \frac d2 \left[ \op{Curl}\ov P \right]^i_\alpha\,\left[ \op{Curl}\ov P \right]^i_\beta\, \mathbb{I}^{\alpha\beta}
\end{align}
which is a quadratic contraction of the tensor \(\op{Curl} P\) -- defined in \autocite{NeffEtAlUnifyingPerspectiveRelaxed2014}. That is, it is proportional to the density of dislocations. Although it appears here as a result of homogenisation, the use of such a term in bidimensional and tridimensional media was already proposed in the relaxed micromorphic model~\cite{NeffEtAlUnifyingPerspectiveRelaxed2014,GhavanlooNeffIsotropicRelaxedMicromorphic2024} -- developed by the last author of the present article -- and subsequently motivated in later articles~\cite{RizziEtAlAnalyticalSolutionsCylindrical2021,SarhilEtAlComputationalApproachIdentify2024}.\\

The problem is not yet entirely formulated on \(\overline{\mathbb{B}}\). Indeed, the tridimensional energy is also composed of bulk and boundary forces, integrated along \( \mathbb{B}\) and \(\partial\mathbb{B}\). The former can be partially integrated as above. For the latter, one has:
\begin{align}
    \label{eq_boundary_beam}
    \partial \mathbb{B} &=
        \{0, L\} \times \transverse \quad \sqcup \quad  \intoo 0L \times \partial \transverse\,.
\end{align}

Boundary forces therefore see their integral split into an integral on the longitudinal boundary \(\times\) transversal bulk -- this is the first term in \cref{eq_boundary_beam}, which yields a homogenised boundary force -- and an integral on the longitudinal bulk \(\times\) transversal boundary -- this is the second term in \cref{eq_boundary_beam}, which yields a homogenised bulk force. The model obtained this way has a unidimensional material, but a tridimensional ambient space. We call it the 1D in 3D model, in contrast with the 3D in 3D model from which it was derived.\\

The actual value of the homogenised forces in terms of the forces of the tridimensional model is irrelevant for this article. The main point is that this correspondence is surjective. It means that one can treat the 1D in 3D case as a beam model on its own -- without any care for what the 3D in 3D constitutive variables are -- and interpret it as coming from a 3D in 3D beam. We therefore abusively denote the external energy as:
\begin{align}
    \label{eq_energy_non_holo_1d_boundary}
    &\int_{\overline{\mathbb{B}}}
      f_0 \cdot \ov{u}
    + f_1 : \ov{P}
    + f_2 \tridots \ov{N}\, \mathrm{d}X^1
    + \int_{\partial \overline{\mathbb{B}}} T_0 \cdot \ov{u} + T_1 : \ov{P} + T_2 \tridots \ov{N}\, \mathrm{d} X^1,
\end{align}

where the forces in \cref{eq_energy_non_holo_1d_boundary} are actually obtained from the above-mentioned procedure and are therefore technically different from those of \cref{eq_energy_1d_raw}. This is not a problem as the 3D in 3D case will not be use further in this paper.

\subsection*{The semi-holonomic case}
\label{sec_1d_in_3d_semi_holo}

In the semi-holonomic case, one has \(N^i_{jk} = P^i_{j,k}\) but \(P^i_j\) and \(u^i_{,j}\) are still independent. This implies:
\begin{align}
    N^i_{j\alpha,k} &= P^i_{j,\alpha k} = P^i_{j,k \alpha} = N^i_{jk,\alpha} = 0\,.
\end{align}
This remark implies that one has:
\begin{align}
    P^i_j\left( X^1, X^2, X^3 \right) &= \overline{P}^i_j\left( X^1\right) + \overline{N}^i_{j\alpha} X^\alpha,\\\nonumber
    u^i\left( X^1, X^2, X^3 \right) &= \overline{u}^i\left( X^1 \right) + \overline{P}^i_\alpha\left( X^1 \right) X^\alpha + \overline{N}^i_{\alpha\beta} \frac {X^\alpha\,X^\beta}2,
\end{align}
where \(\overline{N}^i_{j\alpha}\) and \(\overline{N}^i_{\alpha\beta}\) are constants. Substituting \( \ov{P^i_{j,1}}\) for \( \ov{N^i_{j1}}\), the homogenised internal energy then becomes:
\begin{align}
    &\int_{\overline{\mathbb{B}}}
        \frac a4 \left ( \ov{P^i_j} + \ov{P^j_i} \right )\left ( \ov{P^i_j} + \ov{P^j_i} \right )
         + \frac a4 \left (\ov{N^i_{j\alpha}} + \ov{N^j_{i\alpha}} \right )\left (\ov{N^i_{j\beta}} + \ov{N^j_{i\beta}} \right ) \mathrm{I}^{\alpha\beta}
    \\\nonumber
    &\qquad
         + b \,\ov{P^i_{j,1}}\, \ov{P^i_{j,1}}
         + b \,\ov{N^i_{j\beta}}\, \ov{N^i_{j\beta}}
    \\\nonumber
    &\qquad
        + c\, \ov{P^i_{j,11}}\, \ov{P^i_{j,11}}
        + d \left ( \ov{u^i_{,1}} - \ov{P^i_1} \right ) \left ( \ov{u^i_{,1}} - \ov{P^i_1} \right )
        + d \left ( \ov{P^i_{\alpha,1}} - \ov{N^i_{1\alpha}} \right )\left ( \ov{P^i_{\beta,1}} - \ov{N^i_{1\beta}} \right ) \mathrm{I}^{\alpha\beta}
    \, \mathrm{d}X^1,
\end{align}

In this article, we choose to \emph{focus on semi-holonomic problems where \(\ov{N}^i_{j\alpha}\) is entirely constrained} by at least one boundary condition. This implies that \(\ov{N}^i_{j\alpha}\) can be considered as a constitutive constant (\ie{} given), rather than a kinematic one (\ie{} to be determined). In particular, since adding a constant to the energy does not change anything, the internal energy can be reformulated into:
\begin{align}
    \label{eq_energy_semi_holo_1d}
    &\int_{\overline{\mathbb{B}}}
        \frac a4 \left ( \ov{P^i_j} + \ov{P^j_i} \right )\left ( \ov{P^i_j} + \ov{P^j_i} \right )
        + b\, P^i_{1,1}\, P^i_{1,1}
        + \left( b + \frac{d\,\ell^4}{12} \right) \ov{P^i_{\alpha,1}}\,\ov{P^i_{\alpha,1}}
    \\\nonumber
    &\qquad
        + c\, \ov{P^i_{j,11}}\, \ov{P^i_{j,11}}
        + d \left ( \ov{u^i_{,1}} - \ov{P^i_1} \right ) \left ( \ov{u^i_{,1}} - \ov{P^i_1} \right )
    \, \mathrm{d}X^1
    \\\nonumber
    &\qquad\qquad- \int_{\overline{\mathbb{B}}}
         \frac{d\,\ell^4}{6}\,\ov{N^i_{1\alpha}}\, \ov{P^i_{\alpha,1}}
    \, \mathrm{d}X^1
    +\mathrm{constant}\,,
\end{align}
where, for simplicity, we assumed the transversal symmetry \(\mathrm{I}^{11} = \mathrm{I}^{22} =: \frac {\ell^4}{12}\). Note that this does not mean that \(\transverse\) has to be a square. If, for example, \(\transverse\) is circular then one can replace \(\frac {\ell^4}{12}\) by \(\frac {\pi\, \ell^4}{32}\), where \(\ell\) would be the transversal diameter. Since \(\ell\) is generic, we shall use \(\frac{\ell^4}{12}\) without any loss of generality.\\

When \(P \in \mathfrak{so}(3)\) is antisymmetric -- that is, when the transversal slices are infinitesimally rigidly transformed -- one retrieves the kinematics of Cosserat beams. Furthermore, when the material deformation is constrained in the \(\left( X^1, X^2 \right)\) plane, \(P \in \mathfrak{so}(2)\) and one retrieves the kinematic of Timoshenko beams. When \(c = 0\), \(\ell = 0\) and \(f_2 = 0\), this becomes exactly the energy of the Timoshenko beam, a first-order theory where \(b = E\, I\) and \(d = \kappa\,G\, A\).\footnote{\(E\) is Young's elastic modulus. \(I = \mathrm{I}^{\alpha\alpha}\) (summed) is the second moment of transversal area. \(\kappa\) is the Timoshenko shear coefficient, a function of the shape \(\transverse\) with \(\kappa = \frac 56\) for a rectangular transversal shape. \(A\) is the area of the transversal shape \(\transverse\).} An extension to the \(c \neq 0\) case can be seen in \autocite[eq. 33]{WangEtAlMicroScaleTimoshenko2010}. The latter model is more general but reduces to the semi-holonomic model of this paper for \(\left( k_1, k_2, k_3, k_4, k_5 \right) =  \left(c, b, 0, 0, d\right)\), which implies \(l_1=l_2=0\) \ie{} no deviatoric stretch gradient nor symmetric rotation gradient contribution. A similar model can also be obtained from a study of the Timoshenko beam on Pasternak foundation (\ie{} laying on a set of springs with shear interactions)~\autocite{HarizEtAlBucklingTimoshenkoBeam2022}.

\subsection*{The holonomic case}

In the holonomic case, one has \(N^i_{jk} = P^i_{j,k}\) and \(P^i_j = u^i_{,j}\). The latter implies:
\begin{align}
    P^i_{\alpha,11} &= u^i_{,\alpha11} = u^i_{,11\alpha} = P^i_{1,1\alpha} = N^i_{11,\alpha} = 0
\end{align}

Since we only consider the case where \(\ov{N}\) is entirely constrained by at least one boundary condition, we can consider \(\ov{P^i_{\alpha,1}} = \ov{u^i_{\alpha,1}} = \ov{N^i_{\alpha1}}\) as a constitutive constant (\ie{} given), rather then a kinematic one (\ie{} to be determined). Notice that since \(\overline{u^i_{,\alpha11}} = 0\), one has
\begin{align}
    \overline{u^i_{,\alpha}}\left( X^1 \right) = \overline{u^i_{,\alpha}}(0) + \overline{u^i_{,\alpha1}}\, X^1\,.
\end{align}
Once again, since adding a constant to the energy does not change anything, the internal energy can be reformulated into:
\begin{align}
    \int_{\overline{\mathbb{B}}}
        \frac a4 \left ( \overline{u^i_{,j}} + \overline{u^j_{,i}} \right )\left ( \overline{u^i_{,j}} + \overline{u^j_{,i}} \right )
        + b\, \ov{u^i_{,11}}\, \ov{u^i_{,11}}
        + c\, \ov{u^i_{,111}}\, \ov{u^i_{,111}}
        \, \mathrm{d}X^1
        + \mathrm{constant}\,.
\end{align}
Furthermore, since \(\ov N^i_{\alpha1} = \ov P^i_{\alpha,1}\) is constant, \(\ov P^i_\alpha= \ov u^i_{,\alpha}\) is affine. If \(\ov P\) is prescribed on at least one boundary, then \(\ov{u^i_{,\alpha}}\) can be considered as a given constant. The internal energy can then be reformulated into:
\begin{align}
    \label{eq_energy_holo_1d}
    \int_{\overline{\mathbb{B}}}
        a\, \ov{u^i_{,1}}\,\ov{u^i_{,1}}
        + b\, \ov{u^i_{,11}}\, \ov{u^i_{,11}}
        + c\, \ov{u^i_{,111}}\, \ov{u^i_{,111}}
        \, \mathrm{d}X^1
        + \mathrm{constant}\,.
\end{align}

When the material's deformation is constrained to be in the \(\left( X^1, X^2 \right)\) plane and \(\nabla u = P \in \mathfrak{so}(2)\) is antisymmetric -- that is, when the transversal slices are rigidly transformed -- one obtains the kinematic of the Euler-Bernoulli beam. When \(c=0\), this is exactly the internal energy of the Euler-Bernoulli beam, where \(b = E\, I\) (the contribution of \(a\) is killed by the antisymmetry). In \autocite{LurieSolyaevRevisitingBendingTheories2018}, one can find an extension of Euler-Bernoulli beams to strain gradient theory, in which \(c\) is of the order of \(l^2\,b\). In particular, \[b \gg c\] This model is based on an interpretation of the microstructure due to Mindlin \autocite{MindlinMicrostructureLinearElasticity1964} but, as mentioned by the authors, can also be obtained using a simple gradient elasticity theory with surface tension \autocite[eq. 21 p. 389]{Papargyri-BeskouEtAlBendingStabilityAnalysis2003} which leads to the same equations when their surface term \(\ell \to 0\).\\

\subsection*{Defects}

The present model is based on the model developed in \autocite{CrespoEtAlTwoScaleGeometricModelling2024}. One of the main goals of the latter was the development of a model that could exhibit both disclinations and dislocations from the kinematics alone. One may therefore ask what those are in the present beam model.\\

Dislocations are associated with the torsion tensor \(\mathbf{T}\) (see \autocite{EpsteinGeometricalLanguageContinuum2010}) of the material connection of \autocite{CrespoEtAlTwoScaleGeometricModelling2024}. In the current linear setting, it linearises into 
\begin{align}
    \mathbf{T}^i_{jk} &= N^i_{jk} - N^i_{kj}\,,
\end{align}
which is antisymmetric and therefore has only \(9\) independent components. Notice that, in the semi-holonomic case it simplifies to:
\begin{align}
    \mathbf{T}^i_{jk}
        &= \overline{P^i_{j,k}} - \overline{P^i_{k,j}}
        &\text{if \(N = \nabla P\)}\\\nonumber
        &= \epsilon_{jk\ell}\left[ \op{Curl}P \right]^i_\ell
\end{align}
where \(\epsilon\) is the Levi-Civita fully antisymmetric tensor. In the holonomic case, \(\op{Curl}P = \op{Curl}\nabla u = 0\) and no dislocation can exist. Timoshenko and Cosserat beams can therefore be seen as a relaxation of Euler-Bernoulli beams, allowing the presence of dislocations.\\

Disclinations on the other hand are associated with the curvature tensor \(\mathbf{R}\) (see \autocite{EpsteinGeometricalLanguageContinuum2010}) of the material connection of \autocite{CrespoEtAlTwoScaleGeometricModelling2024}. In the current linear setting, it linearises into:
\begin{align}
    \mathbf{R}^i_{jkl} &= N^i_{jk,l} - N^i_{jl,k}\,,&
    \mathbf{R}^i_{j\alpha1} &= N^i_{j\alpha,1}\,,&
    \mathbf{R}^i_{j\alpha\beta} &= 0\,,
\end{align}
which has \(18\) independent components. Notice that, in the semi-holonomic case, this is identically zero. Our model can therefore be seen as a relaxation of the Timoshenko and Cosserat beams, allowing the presence of disclinations.

\section{The variational principle}
\label{sec_variational_principle}

We now look at the associated static equilibrium problem. The latter corresponds to the minimisation of the total energy in terms of the kinematic variables \(\ov{u}\), \(\ov{P}\) and \(\ov{N}\). This minimisation can be expressed as the solution of the associated Euler-Lagrange equations~\cite{GermainMethodePuissancesVirtuelles1973,GermainMethodVirtualPower2020}. These equations are obtained from the statement that the differential of the total energy, with respect to the kinematic variables, is zero. This technique is called variational calculus.\\

In order to formalise it, let \(\mathfrak{K}\) be the set of physically allowed kinematic tuples \((\ov{u}, \ov{P}, \ov{N})\). This set is a subset of the real vector-space \(\overline{\mathbb{B}} \to \mathbb{R}^3 \times \mathbb{R}^{3\times3} \times \mathbb{R}^{3\times3\times3}\). The assumption of the calculus of variation is that \(\mathfrak{K}\) is a \(\mathcal{C}^1\) sub-manifold. This means that infinitesimal variations \((\delta\ov{u}, \delta\ov{P}, \delta\ov{N})\) of the kinematic variables \((\ov{u}, \ov{P}, \ov{N})\) are well-defined. These form a space \(\mathrm{T}_{(\ov{u}, \ov{P}, \ov{N})}\mathfrak{K}\) called the tangent space of \(\mathfrak{K}\) at \((\ov{u}, \ov{P}, \ov{N})\). Let therefore, for the remainder fo this article, denote \emph{the equilibrium configuration} of the system by \((\ov{u}, \ov{P}, \ov{N})\). The energy \(\Psi : \mathfrak{K} \to \mathbb{R}\) is smooth in the kinematic fields. In particular, it means its differential at \((\ov{u}, \ov{P}, \ov{N})\) is a well-defined application \(\mathrm{T}_{(\ov{u}, \ov{P}, \ov{N})}\Psi : \mathrm{T}_{(\ov{u}, \ov{P}, \ov{N})}\mathfrak{K} \to \mathbb{R}\). Calculus of variation states that the total energy is stationary at equilibrium, meaning that \(\mathrm{T}_{(\ov{u}, \ov{P}, \ov{N})}\Psi = \bold 0\). That is, for all infinitesimal variations \((\delta\ov{u}, \delta\ov{P}, \delta\ov{N})\) of the kinematic variables, one has \(\mathrm{T}_{(\ov{u}, \ov{P}, \ov{N})}\Psi \cdot (\delta\ov{u}, \delta\ov{P}, \delta\ov{N}) = 0\). Upon calculation, this yields (c.f. \cref{eq_energy_non_holo_1d,eq_energy_non_holo_1d_boundary}):
\begin{align}
	\forall \left( \delta \ov{u}, \delta \ov{P}, \delta \ov{N} \right) \in \mathrm{T}_{(\ov{u}, \ov{P}, \ov{N})}\mathfrak{K},\hspace{-7em}&\\\nonumber
	0 &= \int_{\overline{\mathbb{B}}}
		d \left( \ov{u^i_{,1}} - \ov{P^i_1} \right) \delta \ov{u^i_{,1}}
        - \left[ f_0 \right]_{i}\,\delta \ov{u^i}~\mathrm{d} X^1
    - \int_{\partial \overline{\mathbb{B}}}
        \left[ T_0 \right]_i\, \delta \ov{u^i}~\mathrm{d} X^1,\\\nonumber
	0 &= \int_{\overline{\mathbb{B}}}
		\frac a2 \left( \ov{P^i_j} + \ov{P^j_i} \right) \delta \ov{P^i_j}
		- d \left( \ov{u^i_{,1}} - \ov{P^i_1} \right) \delta \ov{P^i_{1}} \, \delta^1_i
        + \frac{d\,\ell^4}{12} \left( \ov{P^i_{\alpha,1}} - \ov{N^i_{1\alpha}} \right)\delta \ov{P^i_{\alpha,1}}\\\nonumber
    &\qquad
		+ e \left( \ov{P^i_{j,1}} - \ov{N^i_{j1}} \right) \delta \ov{P^i_{j,1}}
        - \left[ f_1 \right]^{j}_i\,\delta \ov{P^i_j} ~\mathrm{d} X^1
    - \int_{\partial \overline{\mathbb{B}}}
        \left[ T_1 \right]_i^j\, \delta \ov{P^i_j}~\mathrm{d} X^1,\\\nonumber
	0 &= \int_{\overline{\mathbb{B}}}
        \frac {a\,\ell^4}{24} \left( \ov{N^i_{j\alpha}} + \ov{N^j_{i\alpha}} \right)\delta \ov{N^i_{j\alpha}}
		+ b\,\ov{N^i_{jk}}\,\delta \ov{N^i_{jk}}
		+ c\,\ov{N^i_{jk,1}}\,\delta \ov{N^i_{jk,1}}\\\nonumber
    &\qquad
        - \frac{d\,\ell^4}{12}\left( \ov{P^i_{\alpha,1}} - \ov{N^i_{1\alpha}} \right) \delta_1^j\, \delta \ov{N^i_{j\alpha}} 
		- e \left( \ov{P^i_{j,1}} - \ov{N^i_{j1}} \right)\delta \ov{N^i_{j1}}\,\delta^1_k
        + \frac{e\,\ell^4}{12}\, \ov{N^i_{j\alpha,1}} \, \delta \ov{N^i_{j\alpha,1}}\\\nonumber
    &\qquad
        - \left[ f_2 \right]^{jk}_i\,\delta \ov{N^i_{jk}}~\mathrm{d} X^1
    - \int_{\partial \overline{\mathbb{B}}}
        \left[ T_2 \right]_i^{jk}\, \delta \ov{N^i_{jk}}~\mathrm{d} X^1\,.
\end{align}
Using integration by parts, the above become:
\begin{align}
    \label{eq_weak_generic}
	\forall \left( \delta \ov{u}, \delta \ov{P}, \delta \ov{N} \right) \in \mathrm{T}_{(\ov{u}, \ov{P}, \ov{N})}\mathfrak{K},\hspace{-7em}&\\\nonumber
	0 &= \int_{\overline{\mathbb{B}}}
		-d \left( \ov{u^i_{,11}} - \ov{P^i_{1,1}} \right) \delta \ov{u^i}
        - \left[ f_0 \right]_{i}\,\delta \ov{u^i}~\mathrm{d} X^1\\\nonumber
    &\qquad
    + \int_{\partial \overline{\mathbb{B}}}
        d \left( \ov{u^i_{,1}} - \ov{P^i_1} \right) \delta \ov{u^i}
        - \left[ T_0 \right]_i\,\delta \ov{u^i}
        ~\mathrm{d} X^1,\\\nonumber
	0 &= \int_{\overline{\mathbb{B}}}
		\frac a2 \left( \ov{P^i_j} + \ov{P^j_i} \right) \delta \ov{P^i_j}
		- d \left( \ov{u^i_{,1}} - \ov{P^i_1} \right) \delta \ov{P^i_{1}} \, \delta^1_i
        - \frac{d\,\ell^4}{12} \left( \ov{P^i_{\alpha,11}} - \ov{N^i_{1\alpha,1}} \right)\delta \ov{P^i_{\alpha}}\\\nonumber
    &\qquad
		- e \left( \ov{P^i_{j,11}} - \ov{N^i_{j1,1}} \right) \delta \ov{P^i_j}
        - \left[ f_1 \right]^{j}_i\,\delta \ov{P^i_j}
        ~\mathrm{d} X^1\\\nonumber
    &\qquad
    + \int_{\partial \overline{\mathbb{B}}}
        \frac{d\,\ell^4}{12} \left( \ov{P^i_{\alpha,1}} - \ov{N^i_{1\alpha}} \right)\delta \ov{P^i_{\alpha}}
        + e \left( \ov{P^i_{j,1}} - \ov{N^i_{j1}} \right) \delta \ov{P^i_{j}}
        - \left[ T_1 \right]_i^j\, \delta \ov{P^i_j}
        ~\mathrm{d} X^1,\\\nonumber
	0 &= \int_{\overline{\mathbb{B}}}
        b\,\ov{N^i_{jk}}\,\delta \ov{N^i_{jk}}
		- c\,\ov{N^i_{jk,11}}\,\delta \ov{N^i_{jk}}
		- e \left( \ov{P^i_{j,1}} - \ov{N^i_{j1}} \right)\delta \ov{N^i_{j1}}\,\delta^1_k
        - \left[ f_2 \right]^{jk}_i\,\delta \ov{N^i_{jk}}
        ~\mathrm{d} X^1\\\nonumber
    &\qquad
    + \frac {\ell^4}{12}\, \int_{\overline{\mathbb{B}}}
        \frac {a}{2} \left( \ov{N^i_{j\alpha}} + \ov{N^j_{i\alpha}} \right)\delta \ov{N^i_{j\alpha}}
        - d\left( \ov{P^i_{\alpha,1}} - \ov{N^i_{1\alpha}} \right) \delta_1^j\, \delta \ov{N^i_{j\alpha}} 
        - e\, \ov{N^i_{j\alpha,11}} \, \delta \ov{N^i_{j\alpha}}
    ~\mathrm{d} X^1\\\nonumber
    &\qquad
    + \int_{\partial \overline{\mathbb{B}}}
        c\,\ov{N^i_{jk,1}}\,\delta \ov{N^i_{jk}}
        + \frac{e\,\ell^4}{12}\, \ov{N^i_{j\alpha,1}} \, \delta \ov{N^i_{j\alpha}}
        - \left[ T_2 \right]_i^{jk}\, \delta \ov{N^i_{jk}}
        ~\mathrm{d} X^1\,.
\end{align}

For simplicity, we assume in this article that only anchorings are allowed on the boundary. That is, for all \(X_0^1 \in \partial \overline{\mathbb{B}}\), the value at \(X_0^1\) of any kinematic variable appearing in the (boundary contribution of) the energy is either free -- meaning it can theoretically take any value -- or anchored -- meaning it is restricted to a given value.

\subsection*{Non-holonomic Euler-Lagrange equations}

If no constraint is put on the kinematic variables in the bulk \(\overline{\mathbb{B}} \setminus \partial\overline{\mathbb{B}}\), then the tangent space \(\mathrm{T}_{(\ov{u}, \ov{P}, \ov{N})}\mathfrak{K} = \mathbb{R}^3 \times \mathbb{R}^{3\times3} \times \mathbb{R}^{3\times3\times3}\) is maximal. In particular the problem is non-holonomic. This effectively means that one can localise (see \autocite{GonzalezStuartFirstCourseCountinuum2008}) the above equations into:
\makeatletter
\tagsleft@true
\begin{align}
    \label{eq_euler_lagrange_non_holo}
	\left[ f_0 \right]_i &=
		-d \left( \ov{u^i_{,11}} - \ov{P^i_{1,1}} \right)
        &\text{on \(\overline{\mathbb{B}}\)},\\\nonumber
    \left[ f_1 \right]_i^j &= 
		\frac a2 \left( \ov{P^i_j} + \ov{P^j_i} \right)
		- d \left( \ov{u^i_{,1}} - \ov{P^i_1} \right) \delta^1_j
        - \frac{d\, \ell^4}{12} \left( \ov{P^i_{\alpha,11}} - \ov{N^i_{1\alpha,1}} \right) \delta^\alpha_j
		- e \left( \ov{P^i_{j,11}} - \ov{N^i_{j1,1}} \right)
        &\text{on \(\overline{\mathbb{B}}\)},\\\nonumber
    \left[ f_2 \right]_i^{j1} &= 
        b\,\ov{N^i_{j1}}
		- c\,\ov{N^i_{j1,11}}
		- e \left( \ov{P^i_{j,1}} - \ov{N^i_{j1}} \right)
        &\text{on \(\overline{\mathbb{B}}\)},\\\nonumber
    \left[ f_2 \right]_i^{j\alpha} &= 
        b\,\ov{N^i_{j\alpha}}
        - c\,\ov{N^i_{j\alpha,11}}
        + \frac{\ell^4}{12}\left(
        \frac {a}{2} \left( \ov{N^i_{j\alpha}} + \ov{N^j_{i\alpha}} \right)
        - d\left( \ov{P^i_{\alpha,1}} - \ov{N^i_{1\alpha}} \right) \delta_1^j
        - e\, \ov{N^i_{j\alpha,11}}\right)
        &\text{on \(\overline{\mathbb{B}}\)},\\\nonumber
\forall \left( \delta \ov{u}, \delta \ov{P}, \delta \ov{N} \right) \in \mathrm{T}_{(\ov{u}, \ov{P}, \ov{N})}\mathfrak{K},\hspace{-7em}&\\\nonumber
	0 &= \left (d \left( \ov{u^i_{,1}} - \ov{P^i_1} \right)
        - \left[ T_0 \right]_i\right ) \delta \ov{u^i}
        &\text{on \(\partial \overline{\mathbb{B}}\)},\\\nonumber
    0 &=
        \left (e \left( \ov{P^i_{j,1}} - \ov{N^i_{j1}} \right)
        + \frac{d\,\ell^4}{12} \left( \ov{P^i_{\alpha,1}} - \ov{N^i_{1\alpha}} \right) \delta^\alpha_j
        - \left[ T_1 \right]_i^j\right ) \delta \ov{P^i_j}
        &\text{on \(\partial \overline{\mathbb{B}}\)},\\\nonumber
    0 &=
        \left (c\,\ov{N^i_{j1,1}}
        - \left[ T_2 \right]_i^{j1} \right ) \delta \ov{N^i_{j1}}
        &\text{on \(\partial \overline{\mathbb{B}}\)},\\\nonumber
    0 &=
        \left (c\,\ov{N^i_{j\alpha,1}}
        + \frac{e\,\ell^4}{12}\, \ov{N^i_{j\alpha,1}}
        - \left[ T_2 \right]_i^{j\alpha} \right ) \delta \ov{N^i_{j\alpha}}
        &\text{on \(\partial \overline{\mathbb{B}}\)}.
\end{align}
\tagsleft@false
\makeatother

Once a choice of boundary condition is made, the bondary equations -- involving the universally quantified test functions \(\delta \ov{u}\), \(\delta \ov{P}\) and \(\delta \ov{N}\) -- can be further simplified. The equation obtained this way are the so-called Euler-Lagrange equations. Recall that in \cref{sec_1d_in_3d} we assumed that \(N\) was anchored on at least one boundary. This assumption was necessary for the energy of the semi-holonomic case to simplify into the energy we obtained in \cref{eq_energy_semi_holo_1d}, which is very similar to the one of Cosserat and Timoshenko beams. However, it can be disregarded if one is only concerned with the non-holonomic case. Similarly, the holonomic case needs \(P\) to be anchored on at least one boundary for the energy to simplify into the one obtained in \cref{eq_energy_holo_1d}, which is very similar to the energy of the Euler--Bernoulli beam.\\

\subsection*{Semi-holonomic Euler-Lagrange equations}

In the semi-holonomic case, one has \(N^i_{jk} = P^i_{j,k}\). From the reasoning of \cref{sec_1d_in_3d_semi_holo}, one remove all instances of \(N\) in the energy as long as it is fixed on at least one boundary. In particular, this means that the tangent space \(\mathrm{T}_{(\ov{u}, \ov{P})}\mathfrak{K}\) is isomorphic to \(\mathbb{R}^3 \times \mathbb{R}^{3\times3}\) with \(\ov{\delta N^i_{j1}} = \ov{\delta P^i_{k,1}}\) and \(\ov{\delta N^i_{j\alpha}} = 0\). \Cref{eq_weak_generic} then becomes (c.f. \cref{eq_energy_semi_holo_1d}):
\begin{align}
	\forall \left( \delta \ov{u}, \delta \ov{P} \right) \in \mathrm{T}_{(\ov{u}, \ov{P})}\mathfrak{K},\hspace{-7em}&\\\nonumber
	0 &= \int_{\overline{\mathbb{B}}}
		-d \left( \ov{u^i_{,11}} - \ov{P^i_{1,1}} \right) \delta \ov{u^i}
        - \left[ f_0 \right]_{i}\,\delta \ov{u^i}~\mathrm{d} X^1
    + \int_{\partial \overline{\mathbb{B}}}
        d \left( \ov{u^i_{,1}} - \ov{P^i_1} \right) \delta \ov{u^i}
        - \left[ T_0 \right]_i\,\delta \ov{u^i}
        ~\mathrm{d} X^1,\\\nonumber
	0 &= \int_{\overline{\mathbb{B}}}
		\frac a2 \left( \ov{P^i_j} + \ov{P^j_i} \right) \delta \ov{P^i_j}
        + b\,\ov{P^i_{j,1}}\,\delta \ov{P^i_{j,1}}
		- c\,\ov{P^i_{j,111}}\,\delta \ov{P^i_{j,1}}
		- d \left( \ov{u^i_{,1}} - \ov{P^i_1} \right) \delta \ov{P^i_{1}} \, \delta^1_i
        \\\nonumber
    &\qquad
        - \frac{d\,\ell^4}{12} \ov{P^i_{\alpha,11}}\delta \ov{P^i_{\alpha}}
        - \left[ f_2 \right]^{j1}_i\,\delta \ov{P^i_{j,1}}
        - \left[ f_1 \right]^{j}_i\,\delta \ov{P^i_j}
        ~\mathrm{d} X^1\\\nonumber
    &\quad
    + \int_{\partial \overline{\mathbb{B}}}
        c\,\ov{P^i_{j,11}}\,\delta \ov{P^i_{j,1}}
        + \frac{d\,\ell^4}{12} \ov{P^i_{\alpha,1}}\delta \ov{P^i_{\alpha}}
        - \left[ T_1 \right]_i^j\, \delta \ov{P^i_j}
        - \left[ T_2 \right]_i^{j1}\, \delta \ov{P^i_{j,1}}
        ~\mathrm{d} X^1\,.
\end{align}

Using another integration by part and then localising, this yields the following Euler-Lagrange equations:
\makeatletter
\tagsleft@true
\begin{align}
    \label{eq_euler_lagrange_semi_holo}
	\left[ f_0 \right]_i &=
		-d \left( \ov{u^i_{,11}} - \ov{P^i_{1,1}} \right)
        &\text{on \(\overline{\mathbb{B}}\)},\\\nonumber
    \left[ \widetilde{f}_1 \right]^{j}_i
    &=
        \frac a2 \left( \ov{P^i_j} + \ov{P^j_i} \right)
        - b\,\ov{P^i_{j,11}}
        + c\,\ov{P^i_{j,1111}}
        - d \left( \ov{u^i_{,1}} - \ov{P^i_1} \right) \delta^1_j
        - \frac{d\,\ell^4}{12}\, \ov{P^i_{\alpha,11}}\delta^\alpha_j
        &\text{on \(\overline{\mathbb{B}}\)},\\\nonumber
\forall \left( \delta \ov{u}, \delta \ov{P} \right) \in \mathrm{T}_{(\ov{u}, \ov{P})}\mathfrak{K},&\\\nonumber
        0 &= \left (d \left( \ov{u^i_{,1}} - \ov{P^i_1} \right)
            - \left[ T_0 \right]_i\right ) \delta u^i
            &\text{on \(\partial \overline{\mathbb{B}}\)},\\\nonumber
        0 &= \left (
            b\,\ov{P^i_{j,1}}
            + \frac{d\,\ell^4}{12}\, \ov{P^i_{\alpha,1}}\delta^\alpha_j
            - c\,\ov{P^i_{j,111}}
            - \left[ \widetilde{T}_1 \right]_i^j
            \right ) \delta P^i_j
            &\text{on \(\partial \overline{\mathbb{B}}\)},\\\nonumber
        0 &= \left (
            c\,\ov{P^i_{j,11}}
            - \left[ T_2 \right]_i^{j1}
            \right ) \delta \ov{P^i_{j,1}}
            &\text{on \(\partial \overline{\mathbb{B}}\)}\,,
\end{align}
\tagsleft@false
\makeatother
where we renamed the forces to absorb the redundant terms coming from the homogenisation. Notice that, in this case, the terms in \(\ell^4\) do not greatly impact the qualitative behaviour of the system. Here too, \(P\) has to be anchored on at least one boundary if one wants the energy of the holonomic case to simplify into the energy of \cref{eq_energy_holo_1d}, which is close to the one of Euler--Bernoulli beam.

\subsection*{Holonomic Euler-Lagrange equations}

In the holonomic case, one has \(N^i_{jk} = P^i_{j,k}\) and \(P^i_j = u^i_{,j}\). The reasoning of \cref{sec_1d_in_3d} allows to remove all terms in \(P\) in the energy, as long as it is fixed on at least one boundary. Using the same technique as above, the Euler-Lagrange equations then become:
\begin{align}
    \label{eq_euler_lagrange_holo}
    \left[ \widehat{f}_0 \right]_i &=
        a\,{\ov u^i}^{(2)}
        - b\,{\ov u^i}^{(4)}
        + c\,{\ov u^i}^{(6)}
        &\text{on \(\overline{\mathbb{B}}\)},\\\nonumber
    0 &= \left (
        a\,{\ov u^i}^{(1)}
        - b\,{\ov u^i}^{(3)}
        + c\,{\ov u^i}^{(5)}
        - \left[ \widehat{T}_0 \right]_i \right ) \delta {\ov u^i}
        &\text{on \(\partial \overline{\mathbb{B}}\)},\\\nonumber
    0 &= \left (
        b\,{\ov u^i}^{(2)}
        - c\,{\ov u^i}^{(4)}
        - \left[ \widetilde{T}_1 \right]_i^1 \right ) \delta {\ov u^i}^{(1)}
        &\text{on \(\partial \overline{\mathbb{B}}\)},\\\nonumber
    0 &= \left (
        c\,{\ov u^i}^{(3)}
        - \left[ \widetilde{T}_2 \right]_i^{11} \right ) \delta {\ov u^i}^{(2)}
        &\text{on \(\partial \overline{\mathbb{B}}\)},
\end{align}

where, once again, the forces were renaimed to absorb redundant terms. Such distinction between the forces not being important for the remainder of this article, they will be (slightly abusively) named the same.

\section{Study of the system of ordinary differential equations}
\label{sec_study_system}

The system of ODEs given in the above section is quite complex as it involves \(3 + 3^2 + 3^3 = 39\) scalar kinematic fields. However, several things can be said about it. First, the system is linear, of second order and with constant coefficients. This means that all solutions of the system can be obtained with a good precision by most numerial solvers. Then, not all 39 scalar fields interact with each other. On the contrary, the system can be decomposed into mutiple independent sub-systems. Finally, the role of some of the constitutive coefficients on the solutions can be explicited. We start with the second point.

\subsection{The purely macroscopically-induced beam}
\label{sec_purely_macroscopically_induced_cantilever_beam}

By construction, a completely free system has a trivial solution. Non-trivial solution arise from external solicitations of the system, through either non-zero boundary conditions or non-zero forces. From those a kinematic component can be excited in two ways: directly -- when its boundary condition or a force acting on it is modified -- or indirectly -- when another component is excited directly but couplings in the internal energy force the component to be non-trivial in order to minimise the energy.\\

An interesting object of study when looking at indirect excitation is the dependency graph. A graph where the nodes are the kinematic scalar components and their derivatives, and two nodes \(a\) and \(b\) share an edge when one is the derivative of the other or when the internal energy \(\Psi_{\mathrm{int}}\) satisfies \(\frac{\partial^2 \Psi_{\mathrm{int}}}{\partial a \partial b} \neq 0\). Importantly, the connected components of the graph can be seen as independent ODE systems. That is, a kinematic variable can be excited only by a direct excitation of a kinematic variable in the same connected component. In order to understand indirect excitations in the presented model, we therefore propose to look at the ODE system arising from (the connected components of) the macroscopic fields \(u^i\).

\subsubsection*{The problem of pure macroscopic traction}

The component of \(u^1\) represents the longitudinal elongation. The ODE system induced by its connected component therefore describes a problem of pure traction. A direct observation shows that this connected components is:
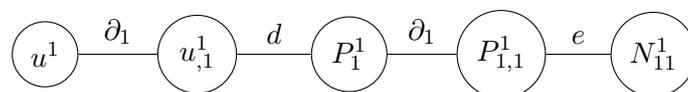
\begin{center}
\begin{tikzpicture}[node distance=2cm, auto]
    \node[draw, circle] (u1) {\(u^1\)};
    \node[draw, circle, right of=u1] (u11) {\(u^1_{,1}\)};
    \node[draw, circle, right of=u11] (P11) {\(P^1_1\)};
    \node[draw, circle, right of=P11] (P111) {\(P^1_{1,1}\)};
    \node[draw, circle, right of=P111] (N111) {\(N^1_{11}\)};

    \draw[-] (u1) -- (u11) node[midway, above] {\(\partial_1\)};
    \draw[-] (u11) -- (P11) node[midway, above] {\(d\)};
    \draw[-] (P11) -- (P111) node[midway, above] {\(\partial_1\)};
    \draw[-] (P111) -- (N111) node[midway, above] {\(e\)};
\end{tikzpicture}
\end{center}

where the edges \(\mathfrak{X}\) --- \(\mathfrak{Y}\) are labelled with \(-\frac 12\frac{\partial^2 \Psi_{\mathrm{int}}}{\partial \mathfrak{X} \partial \mathfrak{Y}}\). The associated Euler-Lagrange equations in the non-holonomic, semi-holonomic and holonomic cases are provided in \cref{app_pure_macroscopic_traction}. Regarding the material defects, this excitation generates no disclinations -- as those are driven by \(N^i_{j\alpha}\) -- and no dislocations -- as those are driven by \(N^i_{jk} - N^i_{kj}\).

\subsubsection*{The problem of pure macroscopic bending}

The component of \(u^2\) represents the transverse displacement. The ODE system induced by its connected component therefore describes a problem of pure (planar) bending. A direct observation shows that this connected components is:

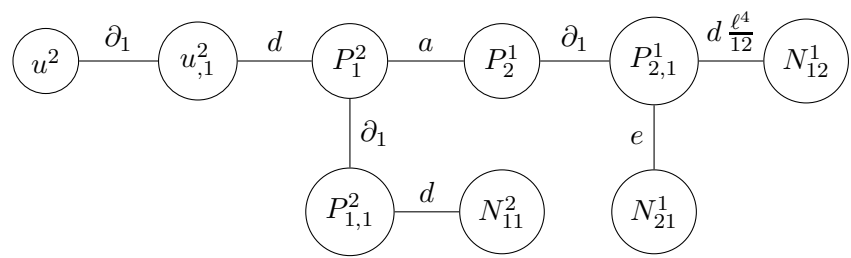
\begin{center}
    \begin{tikzpicture}[node distance=2cm, auto]
        \node[draw, circle] (u2) {\(u^2\)};
        \node[draw, circle, right of=u2] (u21) {\(u^2_{,1}\)};
        \node[draw, circle, right of=u21] (P21) {\(P^2_1\)};
        \node[draw, circle, right of=P21] (P12) {\(P^1_2\)};
        \node[draw, circle, right of=P12] (P121) {\(P^1_{2,1}\)};
        \node[draw, circle, below of=P121] (N121) {\(N^1_{21}\)};
        \node[draw, circle, left of=N121] (N211) {\(N^2_{11}\)};
        \node[draw, circle, left of=N211] (P211) {\(P^2_{1,1}\)};
        \node[draw, circle, right of=P121] (N112) {\(N^1_{12}\)};
    
        \draw[-] (u2) -- (u21) node[midway, above] {\(\partial_1\)};
        \draw[-] (u21) -- (P21) node[midway, above] {\(d\)};
        \draw[-] (P21) -- (P12) node[midway, above] {\(a\)};
        \draw[-] (P12) -- (P121) node[midway, above] {\(\partial_1\)};
        \draw[-] (P121) -- (N121) node[midway, left] {\(e\)};
        \draw[-] (N211) -- (P211) node[midway, above] {\(d\)};
        \draw[-] (P211) -- (P21) node[midway, right] {\(\partial_1\)};
        \draw[-] (P121) -- (N112) node[midway, above] {\(d\,\frac{\ell^4}{12}\)};
    \end{tikzpicture}
    \end{center}

By symmetry, the connected component of \(u^3\) is the same with every \(2\) replaced by a \(3\). The associated Euler-Lagrange equations in the non-holonomic, semi-holonomic and holonomic cases are provided in \cref{app_pure_macroscopic_bending}. Regarding the materials defects, this excitation may generate some dislocations -- from \(N^1_{12} - N^1_{21}\) -- and some disclinations -- from \(N^1_{12}\). Notice that the latter is only excited thanks to the interaction provided by \(d\, \frac{\ell^4}{12}\), the factor of the curl-like term \(P^1_{2,1} - N^1_{12}\).

\subsection{Asymptotic behaviours and the role of the constitutive coefficients}

We now propose to investigate the role of some constitutive coefficients on the solutions of the system.

\subsubsection*{The non-holonomic penality}

The consitutive coefficient \(e\) is a factor, in the tridimensional energy, of the term \(\left\| N - \nabla P \right\|^2\). One can therefore see it as a penalty term, forcing \(N \simeq \nabla P\). One hence expects that, when \(e \to \infty\), \(\left\| N - \nabla P \right\| \to 0\). Indeed, using \cref{eq_euler_lagrange_non_holo}.3, one has:
        \begin{align}
            \ov{P^i_{j,1}} &= 
                \ov{N^i_{j1}}
                + \frac 1e \left (b\,\ov{N^i_{j1}}
                - c\,\ov{N^i_{j1,11}}
                - \left[ f_2 \right]_i^{j1} \right )\,.
        \end{align}
        Using this equality with \cref{eq_euler_lagrange_non_holo}.2, one obtains a system of ODE in \(\ov{u}\) and \(N\) which, upon combining the terms of higher orders and dividing the last equation by e, takes the form:
        \begin{align}
            d\, u^i_{,11} &= \textit{lower-order terms}\,,\\\nonumber
            \frac{d\,\ell^4}{12}\delta^\alpha_j\,\frac ce \, N^i_{\alpha1,111} + c\, N^i_{j1,111} &= \textit{lower-order terms}\,,\\\nonumber
            \left( \frac ce + \frac{\ell^4}{12} \right) N^i_{j\alpha,11} &= \textit{lower-order terms}\,.
        \end{align}
        Importantly, the higher-order terms do not vanish when \(e \to \infty\) nor do the lower order term explode. This means that this system can be expressed as:
        \begin{align}
            \begin{bmatrix}
                K \\ Z \\ W
            \end{bmatrix}^{(1)} &= \begin{bmatrix}
                0 & \mathrm{Id} & 0 \\
                0 & 0 & \mathrm{Id} \\
                \mathbf{A}_0 & \mathbf{A}_1 & \mathbf{A}_2
            \end{bmatrix} \cdot \begin{bmatrix}
                K \\ Z \\ W
            \end{bmatrix} + \mathbf{B}
        \end{align}
        where \(K\) is a vector containing the components of \(\ov{u}\) and \(N\), \(Z\) (resp. \(W\)) is a symbolic substitute for \(K^{(1)}\) (resp. \(K^{(2)}\)) and \(A_i\) and \(B\) are real matrices continuously parametrised by \(\frac 1e\). This system can therefore be seen as a system of first-order linear ODEs with constant coefficients valued in \(\mathcal{C}^0(\mathbb{R}_+^\star, \mathbb{R})\), the set of continuous real-valued fonctions of \(\frac 1e\). The latter is a Hilbert space. Existence and uniqueness of global solutions of the system is then guaranteed once the boundary conditions -- also continuously parametrised by \(\frac 1e\) -- are added. The solution will then be continuously parametrised by \(\frac 1e\). Crucially, it will be well-defined for \(\frac 1e = 0\) and will correspond to the limit of the solution for \(e \to \infty\). In particular, this means that there exists a real number \(C\) such that \(\lim_{e \to \infty} \left\| b\,\ov{N^i_{j1}} - c\,\ov{N}^i_{j1,11} - \left[ f_2 \right]_i^{j1} \right \|_{\mathrm{L}^\infty} \leq C\). Using this, one has:
        \begin{align}
            \lim_{e\to \infty}
                \left\| \ov{N^i_{j1}} - \ov{P^i_{j,1}} \right\|_{\mathrm{L}^\infty}
                &\leq \lim_{e\to \infty}
                    \frac Ce
                = 0\,.
        \end{align}

        This shows that \(\left\| N - \nabla P \right\|_{{\mathrm{L}}^\infty} \underset{e \to \infty} \longrightarrow 0\). In particular, the solutions of the non-holonomic problem converge to the solutions of the corresponding semi-holonomic problem:
        \begin{theo}
            \label{theo_lim_e}
            Let \(\ov{u}_{\mathrm{non-holo}}, \ov{P}_{\mathrm{non-holo}}, \ov{N}_{\mathrm{non-holo}}\) be the solution of the non-holonomic problem (c.f. \cref{eq_euler_lagrange_non_holo}) and \(\ov{u}_{\mathrm{semi-holo}}, \ov{P}_{\mathrm{semi-holo}}, \ov{N}_{\mathrm{semi-holo}}\) be the solution of the semi-holonomic problem (c.f. \cref{eq_euler_lagrange_semi_holo}). If, in both cases, the boundary conditions are indentical, fully specify \(N^i_{j\alpha}\) on (at least one point of) the boundary, and are consistent with \(\nabla P = N\), then one has:
            \begin{align}
                \lim_{e\to \infty}
                    \left\| \ov{u}_{\mathrm{non-holo}} - \ov{u}_{\mathrm{semi-holo}} \right\|_{\mathrm{L}^\infty}
                    &= 0\,,\\\nonumber
                \lim_{e\to \infty}
                    \left\| \ov{P}_{\mathrm{non-holo}} - \ov{P}_{\mathrm{semi-holo}} \right\|_{\mathrm{L}^\infty}
                    &= 0\,,\\\nonumber
                \lim_{e\to \infty}
                    \left\| \left[ \ov{N}_{\mathrm{non-holo}} \right]^i_{j1} - \left[ \ov{P}_{\mathrm{semi-holo}} \right]^i_{j,1} \right\|_{\mathrm{L}^\infty}
                    &= 0\,.
            \end{align}
        \end{theo}

        \begin{proof}
            From the above reasonning, the solution of the non-holonomic problem converges to a well-defined solution of the limiting ODE system. This ODE system is exactly the one of the semi-holonomic problem. This can be checked manually but also comes from the process by which those where obtained. Indeed, the non-holonomic weak-form takes the following form:
            \begin{align}
                0 &= \int_{\overline{\mathbb{B}}}
                \mathrm{A}_1(\ov{u}, \ov{P}) \, \delta \ov{P} + e \left( N - \ov{P^{(1)}} \right)^{(1)} \, \delta \ov{P}
                + \mathrm{A}_2(\ov{P}, N) \, \delta N + e \left( N - \ov{P^{(1)}} \right) \, \delta N\, \mathrm{d} X^1\,
            \end{align}
            where \(\mathrm{A}_1\) and \(\mathrm{A}_2\) are linear in their arguments.\footnote{\(A_1\) and \(A_2\) are not equal to the earlier \(\mathbf{A}_1\) and \(\mathbf{A}_2\).} After localisation, by differentiating the second and substracting it from the first, one obtains:
            \begin{align}
                \label{eq_proof_theo_lim_e_1}
                0 &= \mathrm{A}_1(\ov{u}, \ov{P}) - \mathrm{A}_2(\ov{P}, N)^{(1)}\,.
            \end{align}
            Parallely, the semi-holonomic case weak formulation is:
            \begin{align}
                0 &= \int_{\overline{\mathbb{B}}}
                    \mathrm{A}_1(\ov{u}, \ov{P}) \, \delta \ov{P} + e \left( N - \ov{P}^{(1)} \right)^{(1)} \, \delta \ov{P}
                    + \mathrm{A}_2(\ov{P}, \ov{P}) \, \delta \ov{P}^{(1)} + e \left( N - \ov{P}^{(1)} \right) \, \delta \ov{P}^{(1)}\, \mathrm{d} X^1\,.
            \end{align}
            Integrating by part and localising yields:
            \begin{align}
                0 &= \mathrm{A}_1(\ov{u}, \ov{P}) - \mathrm{A}_2(\ov{P}, \ov{P})^{(1)}
            \end{align}
            which is exactly the limit of \cref{eq_proof_theo_lim_e_1} when \(N \to \ov{P}\).
        \end{proof}

\subsubsection*{The semi-holonomic penality}

The consitutive coefficient \(d\) is a factor, in the tridimensional energy, of the term \(\left\| P - \nabla u \right\|^2\). One can therefore see it as a penalty term, forcing \(P \simeq \nabla u\). One hence expects that, when \(d \to \infty\), \(\left\| P - \nabla u \right\| \to 0\). Indeed, \cref{eq_euler_lagrange_non_holo}.2 can be rewritten as:
    \begin{align}
        u^i_{,1} &= P^i_1 + \frac 1d \left ( \frac a2 \left( P^i_1 + P^1_i \right) - e \left( P^i_{1,11} - N^i_{11,1} \right) - \left[ f_1 \right]^{1}_i \right )\\\nonumber
        N^i_{1\alpha} &= P^i_{\alpha,1} - \frac {12}{d\, \ell^4} \left ( b\, N^i_{1\alpha} - c\, N^i_{1\alpha,11} - \left[ f_2 \right]_i^{j\alpha} \right ) + \frac 1d \left ( \frac a2 \left( N^i_{j\alpha} + N^j_{i\alpha} \right) - e\, N^i_{j\alpha,11} \right ) 
    \end{align}
    Substituting these into the other equations of \cref{eq_euler_lagrange_non_holo} yields a system of ODE whose parametrisation in \(\frac 1d\) is well-defined when \(\frac 1d = 0\). Using the exact same arguments as above, this means that \(\left \| \nabla u - P \right \|_{\mathrm{L}^\infty} \underset{d \to \infty}{\longrightarrow} 0\) and \(\left \| N^i_{1\alpha} - P^i_{\alpha,1} \right \|_{\mathrm{L}^\infty} \underset{d \to \infty}{\longrightarrow} 0\). In particular, this means that the solutions of the semi-holonomic problem converge to the solutions of the corresponding holonomic problem:
    \begin{theo}
        Let \(\ov{u}_{\mathrm{semi-holo}}, \ov{P}_{\mathrm{semi-holo}}\) be the solution of the semi-holonomic problem (c.f. \cref{eq_euler_lagrange_semi_holo}) and \(\ov{u}_{\mathrm{holo}}\) be the solution of the holonomic problem (c.f. \cref{eq_euler_lagrange_holo}). If, in both cases, the boundary conditions are indentical, fully specify \(P^i_{\alpha}\) on (at least one point of) the boundary, and are consistent with \(\nabla u = P\), then one has:
        \begin{align}
            \lim_{d\to \infty}
                \left\| \ov{u}_{\mathrm{semi-holo}} - \ov{u}_{\mathrm{holo}} \right\|_{\mathrm{L}^\infty}
                &= 0\,,\\\nonumber
            \lim_{d\to \infty}
                \left\| \left[ \ov{P}_{\mathrm{semi-holo}}  \right]^i_1- \left[ \ov{u}_{\mathrm{holo}} \right]^i_{,1} \right\|_{\mathrm{L}^\infty}
                &= 0\,.
        \end{align}
    \end{theo}

    The proof is the same as for \cref{theo_lim_e} and relies on the fact that the limiting ODE system of the nono-holonomic problem when \(d \to \infty\) is exactly the one of the holonomic problem.

\section{Conclusion}

In this paper, we have presented a family of beam models derived from a three-dimensional higher-order elasticity framework. By introducing three kinematic fields -- the macroscopic displacement \(u\), the micro-distortion tensor \(P\), and the third-order tensor \(N\) -- we have systematically explored three regimes: holonomic, semi-holonomic, and non-holonomic. Each regime corresponds to different levels of kinematic constraints, ranging from classical elasticity to fully relaxed micromorphic models.

We have shown that the holonomic case reduces to a higher-order Euler--Bernoulli beam model, while the semi-holonomic case generalises the Timoshenko beam model. The non-holonomic case, which allows for the most general kinematics, provides a unified framework that includes both dislocations and disclinations as natural outcomes of the model.

Furthermore, we have demonstrated that the holonomic and semi-holonomic models can be recovered as singular limits of the non-holonomic model by increasing the penalty coefficients \(d\) and \(e\), respectively. This provides a clear hierarchical structure among the models and shows that the proposition can be seen as one continuous framework -- parametrised by the amount of relaxation -- rather than three different models.

The simplified ODE systems derived for specific cases, such as pure traction and bending, illustrate the practical applicability of the models and shows how, even in the simple case of bending, both types of defects can arise.

Future work will focus on better understanding the behaviour of the framework in some simple excitation regimes, investigate the physical relevence of some constraints (\eg{} transversal rigidity) in the hope of simplifying the model further, look for analytical solutions in some of those case and analyse the behaviours of defects in different regimes.

\subsection{Acknowledgments}

Mewen Crespo would like to thank the Faculty of Mathematics of the University of Duisburg-Essen for its hospitality during the semester-long on-site collaboration with Patrizio Neff, of which the present work is the fruit. Furthermore, the authors would like to thank the CNRS' research group GDR GDM CNRS for stimulating the interactions between differential geometry and mechanics. An interaction which is at the heart of this work. They would also like to thank CNRS, Univ Rennes, ANR-11-LABX-0020-0 program Henri Lebesgue Center for their financial support.

\section*{Appendixes}
\appendix
\section{Euler-Lagrange equations -- pure macroscopic traction}
\label{app_pure_macroscopic_traction}

\subsubsection*{non-holonomic pure macroscopic traction problem}

\begin{align}
	\left[ f_0 \right]_1 &=
		-d \left( u^1_{,11} - P^1_{1,1} \right)
        &\text{on \(\overline{\mathbb{B}}\)},\\\nonumber
    \left[ f_1 \right]_1^1 &= 
		a\, P^1_1
		- d \left( u^1_{,1} - P^1_1 \right)
		- e \left( P^1_{1,11} - N^1_{11,1} \right)
        &\text{on \(\overline{\mathbb{B}}\)},\\\nonumber
    \left[ f_2 \right]_1^{11} &= 
        b\,N^1_{11}
		- c\,N^1_{11,11}
		- e \left( P^1_{1,1} - N^1_{11} \right)
        &\text{on \(\overline{\mathbb{B}}\)},\\\nonumber
\forall \left( \delta \ov {u^1}, \delta \ov {P^1_1}, \delta \ov {N^1_{11}} \right) \in \mathrm{T}_{(\ov {u^1}, \ov {P^1_1}, \ov {N^1_{11}})}\mathfrak{K},\hspace{-5em}&\\\nonumber
	0 &= \left (d \left( u^1_{,1} - P^1_1 \right)
        - \left[ T_0 \right]_1\right ) \delta u^1
        &\text{on \(\partial \overline{\mathbb{B}}\)},\\\nonumber
    0 &=
        \left (e \left( P^1_{1,1} - N^1_{11} \right)
        - \left[ T_1 \right]_1^1\right ) \delta P^1_1
        &\text{on \(\partial \overline{\mathbb{B}}\)},\\\nonumber
    0 &=
        \left (c\,N^1_{11,1}
        - \left[ T_2 \right]_i^{j1} \right ) \delta N^1_{11}
        &\text{on \(\partial \overline{\mathbb{B}}\)}.
\end{align}

\subsubsection*{semi-holonomic pure macroscopic traction problem}

\begin{align}
	\left[ f_0 \right]_1 &=
		-d \left( u^1_{,11} - P^1_{1,1} \right)
        &\text{on \(\overline{\mathbb{B}}\)},\\\nonumber
    \left[ \widetilde{f}_1 \right]^{j}_i
    &=
        a\, \ov{P^1_1}
        - b\,\ov{P^1_{1,11}}
        + c\,\ov{P^1_{1,1111}}
        - d \left( \ov{u^1_{,1}} - \ov{P^1_1} \right)
        &\text{on \(\overline{\mathbb{B}}\)},\\\nonumber
\forall \left( \delta \ov {u^1}, \delta \ov {P^1_1} \right) \in \mathrm{T}_{(\ov {u^1}, \ov {P^1_1})}\mathfrak{K},&\\\nonumber
        0 &= \left (d \left( u^1_{,1} - P^1_1 \right)
            - \left[ T_0 \right]_1\right ) \delta u^1
            &\text{on \(\partial \overline{\mathbb{B}}\)},\\\nonumber
        0 &= \left (
            b\,\ov{P^1_{1,1}}
            - c\,\ov{P^1_{1,111}}
            - \left[ \widetilde{T}_1 \right]_1^1
            \right ) \delta P^1_1
            &\text{on \(\partial \overline{\mathbb{B}}\)},\\\nonumber
        0 &= \left (
            c\,\ov{P^1_{1,11}}
            - \left[ T_2 \right]_1^{11}
            \right ) \delta \ov{P^1_{1,1}}
            &\text{on \(\partial \overline{\mathbb{B}}\)}\,.
\end{align}

\subsubsection*{holonomic pure macroscopic traction problem}

\begin{align}
    \label{pure_macroscopic_traction_holonomic}
    \left[ \widetilde{f}_0 \right]_1 &=
        a\,{u^1}^{(2)}
        - b\,{u^1}^{(4)}
        + c\,{u^1}^{(6)}
        &\text{on \(\overline{\mathbb{B}}\)},\\\nonumber
    0 &= \left (
        a\,{u^1}^{(1)}
        - b\,{u^1}^{(3)}
        + c\,{u^1}^{(5)}
        - \left[ \widetilde{T}_0 \right]_1 \right ) \delta u^1
        &\text{on \(\partial \overline{\mathbb{B}}\)},\\\nonumber
    0 &= \left (
        b\,{u^1}^{(2)}
        - c\,{u^1}^{(4)}
        - \left[ \widetilde{T}_1 \right]_1^1 \right ) \delta {u^1}^{(1)}
        &\text{on \(\partial \overline{\mathbb{B}}\)},\\\nonumber
    0 &= \left (
        c\,{u^1}^{(3)}
        - \left[ \widetilde{T}_2 \right]_1^{11} \right ) \delta {u^1}^{(2)}
        &\text{on \(\partial \overline{\mathbb{B}}\)}\,.
\end{align}

\section{Euler-Lagrange equations -- pure macroscopic bending}
\label{app_pure_macroscopic_bending}

\subsubsection*{non-holonomic pure macroscopic bending problem}

\begin{align}
	\left[ f_0 \right]_2 &=
		-d \left( u^2_{,11} - P^2_{1,1} \right)
        &\text{on \(\overline{\mathbb{B}}\)},\\\nonumber
    \left[ f_1 \right]_2^1 &= 
		\frac a2 \left( P^1_2 + P^2_1 \right)
		- d \left( u^2_{,1} - P^2_1 \right)
		- e \left( P^2_{1,11} - N^2_{11,1} \right)
        &\text{on \(\overline{\mathbb{B}}\)},\\\nonumber
    \left[ f_1 \right]_1^2 &= 
        \frac a2 \left( P^1_2 + P^2_1 \right)
        - \frac{d\, \ell^4}{12} \left( P^1_{2,11} - N^1_{12,1} \right)
        - e \left( P^1_{2,11} - N^1_{21,1} \right)
        &\text{on \(\overline{\mathbb{B}}\)},\\\nonumber
    \left[ f_2 \right]_2^{11} &= 
        b\,N^2_{11}
        - c\,N^2_{11,11}
        - e \left( P^2_{1,1} - N^2_{11} \right)
        &\text{on \(\overline{\mathbb{B}}\)},\\\nonumber
    \left[ f_2 \right]_1^{21} &= 
        b\,N^1_{21}
        - c\,N^1_{21,11}
        - e \left( P^1_{2,1} - N^1_{21} \right)
        &\text{on \(\overline{\mathbb{B}}\)},\\\nonumber
    \left[ f_2 \right]_1^{12} &= 
        b\,N^1_{12}
        - c\,N^1_{12,11}
        + \frac{\ell^4}{12}\left(
        \frac {a}{2} \left( N^1_{12} + N^1_{12} \right)
        - d\left( P^1_{2,1} - N^1_{12} \right)
        - e\, N^i_{j2,11}\right)
        &\text{on \(\overline{\mathbb{B}}\)},\\\nonumber
\forall \left( \delta \ov{u}, \delta \ov{P}, \delta \ov{N} \right) \in \mathrm{T}_{(\ov{u}, \ov{P}, \ov{N})}\mathfrak{K},\hspace{-5em}&\\\nonumber
	0 &= \left (d \left( u^2_{,1} - P^2_1 \right)
        - \left[ T_0 \right]_2\right ) \delta u^2
        &\text{on \(\partial \overline{\mathbb{B}}\)},\\\nonumber
    0 &=
        \left (e \left( P^2_{1,1} - N^2_{11} \right)
        - \left[ T_1 \right]_2^1\right ) \delta P^2_1
        &\text{on \(\partial \overline{\mathbb{B}}\)},\\\nonumber
    0 &=
        \left (e \left( P^1_{2,1} - N^1_{21} \right)
        + \frac{d\, \ell^4}{12} \left( P^1_{2,1} - N^1_{12} \right)
        - \left[ T_1 \right]_2^1\right ) \delta P^2_1
        &\text{on \(\partial \overline{\mathbb{B}}\)},\\\nonumber
    0 &=
        \left (c\,N^2_{11,1}
        - \left[ T_2 \right]_2^{11} \right ) \delta N^2_{11}
        &\text{on \(\partial \overline{\mathbb{B}}\)},\\\nonumber
    0 &=
        \left (c\,N^1_{21,1}
        - \left[ T_2 \right]_1^{21} \right ) \delta N^1_{21}
        &\text{on \(\partial \overline{\mathbb{B}}\)},\\\nonumber
    0 &=
        \left (c\,N^1_{12,1}
        + \frac{e\,\ell^4}{12}\, N^1_{12,1}
        - \left[ T_2 \right]_1^{12} \right ) \delta N^1_{12}
        &\text{on \(\partial \overline{\mathbb{B}}\)}.
\end{align}

\subsubsection*{semi-holonomic pure macroscopic bending problem}

\begin{align}
	\left[ f_0 \right]_2 &=
		-d \left( u^2_{,11} - P^2_{1,1} \right)
        &\text{on \(\overline{\mathbb{B}}\)},\\\nonumber
    \left[ \widetilde{f}_1 \right]^{1}_2
    &=
        \frac a2 \left( \ov{P^2_1} + \ov{P^1_2} \right)
        - b\,\ov{P^2_{1,11}}
        + c\,\ov{P^2_{1,1111}}
        - d \left( \ov{u^2_{,1}} - \ov{P^2_1} \right)
        &\text{on \(\overline{\mathbb{B}}\)},\\\nonumber
    \left[ \widetilde{f}_1 \right]^{2}_1
    &=
        \frac a2 \left( \ov{P^1_2} + \ov{P^2_1} \right)
        - b\,\ov{P^1_{2,11}}
        + c\,\ov{P^1_{2,1111}}
        - \frac{d\,\ell^4}{12} \ov{P^1_{2,11}}
        &\text{on \(\overline{\mathbb{B}}\)},\\\nonumber
\forall \left( \delta \ov{u}, \delta \ov{P} \right) \in \mathrm{T}_{(\ov{u}, \ov{P})}\mathfrak{K},&\\\nonumber
        0 &= \left (d \left( u^2_{,1} - P^2_1 \right)
            - \left[ T_0 \right]_2\right ) \delta u^2
            &\text{on \(\partial \overline{\mathbb{B}}\)},\\\nonumber
        0 &= \left (
            b\,\ov{P^2_{j,1}}
            - c\,\ov{P^2_{1,111}}
            - \left[ \widetilde{T}_1 \right]_2^1
            \right ) \delta P^2_1
            &\text{on \(\partial \overline{\mathbb{B}}\)},\\\nonumber
        0 &= \left (
            b\,\ov{P^1_{2,1}}
            + \frac{d\,\ell^4}{12} \ov{P^1_{2,1}}
            - c\,\ov{P^1_{2,111}}
            - \left[ \widetilde{T}_1 \right]_1^2
            \right ) \delta P^1_2
            &\text{on \(\partial \overline{\mathbb{B}}\)},\\\nonumber
        0 &= \left (
            c\,\ov{P^2_{1,11}}
            - \left[ T_2 \right]_2^{11}
            \right ) \delta \ov{P^2_{1,1}}
            &\text{on \(\partial \overline{\mathbb{B}}\)}\\\nonumber
        0 &= \left (
            c\,\ov{P^1_{2,11}}
            - \left[ T_2 \right]_1^{21}
            \right ) \delta \ov{P^1_{2,1}}
            &\text{on \(\partial \overline{\mathbb{B}}\)}\,.
\end{align}

\subsubsection*{holonomic pure macroscopic bending problem}

These are the same equations as the ones for the holonomic pure macroscopic traction problem (c.f. \cref{pure_macroscopic_traction_holonomic}), with \(u^2\) instead of \(u^1\).

\printbibliography
\end{document}